\documentclass[12pt]{article}
\usepackage{verbatim}
\usepackage{amssymb,amsmath,amsthm}
\usepackage{graphicx}
\usepackage{epsfig}
\newtheorem{corollary}{Corollary}[section]

\newtheorem{proposition}[corollary]{Proposition}
\newtheorem{conjecture}[corollary]{Conjecture}
\newtheorem{theorem}[corollary]{Theorem}
\newcommand{\Prob} {{\mathbb P}}
\newcommand{\Z}{{\mathbb Z}} 
\newcommand{\E}{{\mathbb E}}

\newcommand{\C}{{\mathbb C}}

\newcommand{\dist}{{\rm dist}}

\def \Im {{\rm Im}}
\def \Re {{\rm Re}}
\def \p {\partial}

\def \Half {{\mathbb H}}

\def \Disk {{\mathbb D}}

\def \saws {{\mathcal W}}
\def \paths {{\mathcal K}}

\def \looper {{\cal J}}

\def \sdom {{\mathcal S}}

\def \tmass {{\Psi}}

\begin{document}

\title{\sc{The probability that planar loop-erased random walk uses a given edge}     } 
\author{\sc{ 
Gregory F. Lawler} \footnote{Research supported by NSF grant DMS-0907143.}}
\date{\textit{University of Chicago}}
\maketitle

\begin{abstract}
We give a new proof of a result of Rick Kenyon that the probability that
an edge in the middle of an $n \times n$ square is used in a loop-erased
walk connecting opposites sides is of order $n^{-3/4}$.  We, in fact, improve
the result by showing that this estimate is correct up to  
multiplicative constants. 
\end{abstract}

\section{Introduction}

Loop-erased random walk is a process obtained by
erasing loops from simple random walk. Although the
process can be defined for arbitrary Markov chains, we
will discuss the process only on the planar
integer lattice $\Z^2 = \Z + i \Z$.  We start this
paper by stating our main result.

Let 
\[    A_n = \{j+ik \in \Z + i \Z: -n + 1 < j < n,
  - n < k < n\} ,\]
  \[   \p A_n = \{z\in \Z^2: \dist(z,A_n) = 1 )\}.\]  
Let $\paths_n$ denote the set of nearest neighbor
paths
$  \omega = [\omega_0,\ldots,\omega_k] $
with $\Re[\omega_0] = -n, \Re[\omega_k]  = n+1$ and $\{\omega_1,\ldots,
\omega_{k-1}\} \subset A_n$. We write $|\omega| = k$ for the
number of steps, and let $p(\omega) = 4^{-|\omega|}$ be
the simple random walk measure.  Let
\[         f(n)  = \sum_{\omega \in \paths_n} p(\omega) . \]
 It is known that $\lim_{n \rightarrow
 \infty} f(n) =  c_1  \in (0,\infty)$
(see, e.g., \cite[Proposition 8.1.3]{LLimic}), where the constant $c_1$ 
can be given in terms of the Green's function of Brownian motion
 on a domain bounded by a square.

Let $\saws_n$ denote the set of self-avoiding walks (SAWs)
$\eta = [\eta_0,\ldots,\eta_k] \in \paths_n$.
For each $\omega \in \paths_n$ there is a
unique path $L(\omega)
\in \saws_n$ obtained by chronological loop-erasing (see
\cite[Chapter 9]{LLimic} for appropriate definitions).  The
loop-erased measure $\hat p_n(\eta)$ is defined by
\[              \hat p_n(\eta) = \sum_{\omega \in 
\paths_n, L(\omega) = \eta } \; p(\omega) . \]
Note that 
\[                \sum_{\eta \in \saws_n}
  \hat p_n(\eta) = f(n)  . \] 
   Let $\saws_n^+$ denote the
set of $\eta \in \saws_n$ that contain the directed edge
$[0,1]$ and $\saws_n^-$ those that contain $[1,0]$.  Let
$\saws_n^* = \saws_n^- \cup \saws_n^+$ be the set of
$\eta \in \saws_n$ that contain the edge $[0,1]$ in
either direction.  We write $a_n \asymp b_n$
to mean that $a_n/b_n$ and $b_n/a_n$ are uniformly
bounded.   The goal of this paper is to prove
the following theorem.

\begin{theorem} As $n \rightarrow \infty$, 
\begin{equation}  \label{nov20.1}
  \sum_{\eta \in \saws_n^*} \hat p_n(\eta) \asymp n^{-3/4} .
  \end{equation}
             \end{theorem}
             
With
 a little more work, we could establish the existence of
the limit
\[   \lim_{n \rightarrow \infty} n^{3/4}   \sum_{\eta \in \saws_n^*} \hat p_n(\eta),\]
but we will not do it here. (Our argument would not give
the value of the limit.  While we believe we might
be able to some of the relevant
asymptotic  constants, we definitely do not
know how to compute the value of the limit in \eqref{nov20.2} below.)    Our result can be considered a
strengthening of a result of 
  Kenyon  \cite{Kenyon}  who proved that
\begin{equation} \label{kenyonresult}  
\sum_{\eta \in \saws_n^*} \hat p_n(\eta)     
\approx    n^{-3/4}  , 
\end{equation}
 where $\approx$ indicates that the logarithms of both
 sides are asymptotic.  His proof used the relationship between
 loop-erased walks and two other models, dimers and uniform
 spanning trees.   Another proof of \eqref{kenyonresult}
 was given by Masson \cite{Masson} using the relationship
 between loop-erased walk and the Schramm-Loewner evolution
 ($SLE$).   
 We do not need to make   reference to any of these
 models in our proof of \eqref{nov20.1}.
 There are two main steps.
 \begin{itemize}
 \item  A combinatorial identity is proved which writes the left-hand
 side of \eqref{nov20.1} in terms of simple random walk quantities.
 \item  The simple random walk quantities are estimated.  
 \end{itemize}

   Our computation to obtain the exponent $3/4$
 uses the Brownian loop measure to estimate the random walk
 loop measure.  This is in the spirit of Kenyon's calculations \cite{Kenyon} since
 the 
 loop measure is closely related to the determinant of the
 Laplacian.

Although the proof is self-contained (other than some
estimates for  simple
random walk) it does use a key idea 
 from
  Kenyon's work  as discussed in \cite[Section 5.7]{KWilson}.
   For each random walk path $\omega$, we
 let $J(\omega)$ be the number of times that the path crosses
 an edge of the form $[-ki,-ki+1]$ or $[-ki+1,-ki]$ where
 $k$ is a positive integer.  Let $q(\omega) = (-1)^{J(\omega)} \, p(\omega)$.
Let $Y_{+}(\omega)$ denote the number of times that $\omega$
uses  the directed edge $[0,1]$, $Y_-(\omega)$ the number of times that
 $\omega$ uses the directed edge $[1,0]$, and
 $Y(\omega) = Y_+(\omega) - Y_-(\omega)$.
 The combinatorial identity is obtained by 
 writing the quantity
\begin{equation}  \label{Lambda}
          \Lambda_n = \sum_{\omega \in \paths_n}
    q(\omega) \, Y(\omega)
    = \sum_{\omega \in \paths_n}
    p(\omega) \,(-1)^{J(\omega) }\,  Y(\omega) . \end{equation}
   in two different ways.

The paper is written using the perspective of loop-erased walk in
terms of the random walk loop measure as in \cite[Chapter 9]{LLimic}.
We start by reviewing this perspective in Section \ref{loopsec} and then
we prove the identity in Section \ref{identsec}.  Section
\ref{rwloopsec}
discusses the random walk estimates.

After the main result is proved, we generalize the combinatorial
identity to random walk starting at any two points on $\p A_n$.
The argument is essentially the same, and one can compute the
dependence of the probability on the starting points.  A somewhat
analogous 
calculation was  done in \cite{Kenyon}, where the boundary was
fixed but the interior point was allowed to vary.
  In the last section, this
dependence on the boundary point is explained in terms of the scaling
limit of loop-erased random walk, the Schramm-Loewner evolution
with parameter $SLE_2$.  Here we do not give all details.

One of the main motivations for doing the estimates in this paper is
to show that the loop-erased random walk converges to $SLE_2$
in the {\em natural parametrization} \cite{LR,LS}.  
Up-to-constant estimates for the loop-erased walk probability can
be viewed as a
step in the program to establish this result.

\section{Random walk loop measure}  \label{loopsec}

The random walk loop measure is a measure on unrooted random
walk loops.   A rooted loop is a nearest neighbor path
\begin{equation}  \label{nov21.1}
 l = [l_0,l_1,\ldots,l_{2k}] 
 \end{equation}
with $k \geq 0$ and $l_0 = l_{2k}$.  We call $l_0$ the
root of the loop.  An unrooted loop $\bar{l}$
 is an
equivalence class of rooted loops with $k > 0$ under the equivalence
\[     [l_j,l_{j+1},\ldots,l_{2k},l_1,l_2,\ldots,l_j]
 \sim  [l_0,l_1,\ldots,l_{2k}]  \]
 for all $j$.  Note that the orientation of the loops is maintained.
 The random walk loop measure $m$ is defined
 by
 \[  m(\bar l) =   
 4^{-|\bar l|} \,\frac{d(\bar l)}{|\bar l|},\] where 
 $d(\bar l)$ is the number of rooted loops 
 in the equivalence class of
 the unrooted loop $\overline l$.  Note
 that $d(\bar l)$ is always
 an integer dividing $|\bar l|$.   In a slight abuse
 of notation, if $l$ is a loop
 and $A \subset \Z^2$, we write $l \subset A$ to mean
 that the vertices of $l$ are contained in $A$ and
 $l \cap A$ for the set of vertices in $A$ that $l$ visits.

 There  is an  equivalent way 
 of defining this measure that is sometimes
 useful. Enumerate $\Z^2 = \{v_1,v_2,\ldots,\}$
 and let $V_n = \{v_1,\ldots,v_n\}$.  We define
 a different measure on rooted loops by assigning
 to each 
(rooted)  loop as in \eqref{nov21.1} with $ k > 0,
 l \subset V_n ,$ and $l_0 = v_n$  measure $s^{-1} \, 4^{-2k}$
 where $s = \#\{j; 1 \leq j \leq 2k, l_j = v_n\}$. This induces
  a measure on unrooted loops by summing over rooted loops
 that generate an unrooted loop.  (The factor $s^{-1}$ compensates
 for the fact that several rooted loops give the same unrooted loop.) One can check that the induced measure on
 unrooted loops is the same as the loop measure above regardless of 
 which enumeration is chosen.
 For computations it is often convenient to choose
 an enumeration    in which
 $|v_j|$ is nondecreasing.
 
If $V  = \{v_1,\ldots,v_k\} \subset A \subsetneq \Z^2$,   we define
\[  F_V(A) =  \exp \left\{\sum_{\bar l \subset A,
\bar l\cap V \neq \emptyset } m(\bar l)\right\}
= \prod_{j=1}^k G_{U_j}(v_j,v_j) . \]
 Here $U_j = A \setminus \{v_1,\ldots,v_{j-1}\}$ and $G_U $
 denotes the usual random walk Green's function
 in the set $U$. 
  The loop-erased measure satisfies \cite[Proposition 9.5.1]{LLimic}
 \begin{equation}
 \label{jan19.1}
             \hat p_n(\eta)   = p(\eta) \, F_\eta(A_n) . 
             \end{equation}

We can also define a loop measure using the signed
weight $q(\omega) = (-1)^{J(\omega) } \, p(\omega) $.
The quantities $J(l), Y(l)$ as defined in the introduction are
functions of  the unrooted loop $\bar l$. (Note that
$Y(l)$ does depend on the orientation of $l$, so it is important
that we are considering oriented, unrooted loops.)
 Let $\looper_A$ denote
the set of unrooted loops $\bar l \subset A$ such
that $J(\bar l)$ is odd.  If $V \subset A$, let $\looper_{A,J}$ denote
the set of unrooted loops $\bar l \in \looper_A$ that
intersect $V$.
Let
\[   Q_V(A) = \exp\left\{\sum_{\bar l \subset A, \bar l
\cap V \neq \emptyset } m(\bar l) \, (-1)^{J(\bar l)}.
\right\} = F_V(A) \, \exp \left\{-2m(\looper_{A,J})
\right\}. \ 
\]
As in the case for $F$, if $V = \{v_1,\ldots,v_k\}
\subset A$, then
\begin{equation}  \label{jan19.2}
   Q_V(A) = \prod_{j=1}^k  g_{U_j}(v_j,v_j) ,
   \end{equation}
where $U_j = A \setminus \{v_1,\ldots,v_{j-1}\}$ and
\[    g_{U}(v_j,v_j) =  \sum_{l}q(l) = \sum_{l}
   (-1)^{J(l)} \, p(l)\]
where the sum is over all (rooted) loops $l$ from $v_j$ to $v_j$ staying
in $U$.   In particular, if $\eta \in \saws_n$, then
when the algebraic computation which gives \eqref{jan19.1} is  applied
to $q$, we get  
\[  \sum_{\omega \in \paths_n, L(\omega) = \eta}
       q(\omega)
        = q(\eta) \, Q_\eta(A_n) . \]
This implies that 
\[  \sum_{\omega \in \paths_n, L(\omega) = \eta}
     (-1)^{J(\omega) - J(\eta)} \, p(\omega)
        = p(\eta) \,Q_\eta(A_n) . \]

 If $V = \{0\}$ is a singleton set, then
\[  \lim_{n \rightarrow \infty} Q_V(A_n) = Q_V(\Z^2)
   =   \sum_{k=0}^\infty s^k =
   (1- s)^{-1} > 0. \]
Here $s = \E[J']$  where $J' = (-1)^{J(S[0,T_0])}$,
 $S$ is a simple random walk starting at the the origin,
and $T_0 = \min\{n \geq 1: S_n > 0\}$.  Since $\Prob\{J' = 1\}
> 0$ and $\Prob\{J' = - 1\} > 0$, we see 
that $|s| < 1$ and hence the limit exists and is positive.
A similar argument shows that if $\Z^2 \setminus U$ is
finite and non-empty,
 and $v$ is in the unbounded component of $U$, then
$g_U(v,v)$ is finite and strictly positive. 
Given this and \eqref{jan19.2}, 
it is straightforward to show that if $V$ is finite,
then
\begin{equation}  \label{nov25.1}
  Q_V = Q_V(\Z^2) = \lim_{n \rightarrow \infty}
   Q_V(A_n) 
   \end{equation}
 exists and is strictly positive.  We will write
 $Q_{01}(A_n)$ for $Q_{\{0,1\}}(A_n)$

For the important computation of the random walk loop
measure, we will use the Brownian loop measure as 
introduced in \cite{Lsoup}.   There are several equivalent
definitions.  We give one here that is convenient for computational
purposes and is analogous to the second definition
we gave for the random walk loop measure.  We start with the Brownian (boundary) bubble measure
in the upper half plane $\Half$ started at the origin.  It is
the limit as $\epsilon \downarrow 0$ of a measure on paths
from $i \epsilon$ to $0$ in $\Half$ which we now
describe.   For each $\epsilon$ consider
the measureof total mass $\epsilon^{-1}$  on paths whose normalized
probability measure is that of a Brownian  $h$-process to $0$.  (An 
$h$-process can be viewed roughly as a Brownian motion conditioned to
leave $\Half$ at $0$.)  As $\epsilon \downarrow 0$, the
limit measure is a $\sigma$-finite measure $\nu_\Half(0)$
 on loops from $0$
to $0$ otherwise in $\Half$.  The normalization is such that the measure
of bubbles that hit the unit circle equals one.  This definition
can be  extended to simply
connected domains with smooth boundaries either by the analogous definition
or by  the following conformal covariance rule:
 if $f: \Half \rightarrow D$ is a conformal
transformation, then
\[         f \circ \nu_\Half(0) = |f'(0)|^2 \,  \nu_D(f(0)) . \]
(In the definition of $f \circ \nu_\Half(0)$,
we need to modify the parametrization of the curve using
Brownian scaling, but the parametrization is not important in this paper.)
 
  Given the Brownian bubble measure, the Brownian loop measure restricted
 to curves in the unit disk $\Disk$ can be written as
\begin{equation}  \label{loopdefinition}      
 \frac 1 \pi \int_{0}^1 \int_{0}^{2 \pi} \nu_{r\Disk}(re^{i\theta}) \,r\, \, d\theta \, dr, 
 \end{equation}
 To be more precise, the loop measure is the measure on {\em unrooted} loops
 induced by the above measure on rooted loops.  (This representation of the
 measure on unrooted loops focuses on the rooted representative with root
 as far from the origin as possible.)
 The Brownian loop measure is the scaling limit of the random walk
 loop measure in a sense made precise in \cite{Jose}.  We discuss this
 more in Section \ref{rwloopsec}.

 \section{A combinatorial identity}  \label{identsec}
 
 Let $\paths_n'$ denote the set of nearest neighbor paths
 $\omega = [\omega_0,\omega_1,\ldots,\omega_k]$ with $
 \Re[\omega_0] = - n, \omega_k = 0$ and $\{\omega_1,\ldots,
 \omega_{k-1}\} \subset A_n \setminus [0,\infty)$.  Let
 $\paths_n''$ denote the set of nearest neighbor paths
 $\omega = [\omega_0,\omega_1,\ldots,\omega_k]$ with $
 \Re[\omega_0] = n+1, \omega_k = 1$ and $\{\omega_1,\ldots,
 \omega_{k-1}\} \subset  A_n \setminus (-\infty,1]$.  There
 is a natural bijection between $\paths_n'$ and
 $\paths_n''$ obtained by reflection about the 
 line $\{\Re(z) = 1/2\}$.  Let
 \[         R_n = \sum_{\omega \in \paths_n'} p(\omega)
    = \sum_{\omega \in \paths_n''} p(\omega).\]
    Note that 
    $R_n$ equals
$\Prob\{\Re(S_\tau) = -n\}$  where $S$ is a simple random walk 
starting at the origin and $\tau = \min\{j > 0: S_j \in \p A_n
\cup [0,\infty)\}.$  It is known 
(ee e.g., \cite[Proposition 5.3.2]{LLimic})
that
\begin{equation}  \label{nov25.2}
                 R_n  \asymp n^{-1/2} , \,\,\,\,  n \rightarrow
\infty .
\end{equation}

 The goal of this section is to
prove the following combinatorial identity which relates
the  probability   that loop-erased walk uses the 
undirected edge
$\{0,1\}$  to some simple random walk
quantities.

\begin{theorem} \label{combin}
\[         4\sum_{\eta \in \saws_n^*} 
    \hat p_n(\eta) 
    =    Q_{01}(A_n) \, R_n^2 
     \, \exp\left\{ 2 m(\looper_{A_n})\right\}.  
 \]
      \end{theorem}
      
 \begin{proof}

We claim the following:
\begin{equation}  \label{nov19.1}
(-1)^{J(\eta) }
    \,  Y(\eta) = 1  \;\;\mbox{ if }\;\; \eta \in \saws_n^*. 
\end{equation}
To see this consider the path $\eta$ as a continuous path
from $\{\Re(z) = -n\}$ to $\{\Re(z) = n+1\}$
in the domain
$D=\{x+iy \in \C: -n < x < n+1, -n < y < n\}$.
Then $\eta$ is a crosscut of $D$ such that $D \setminus \eta$
consists of two components, the ``top'' component $D^+$
and the ``bottom'' component $D^-$.  Each ordered edge
$[w, w']$ in $\eta$ can be considered as subsets of $\p D^+$
and $\p D^-$.  As we traverse from $w$ to $w'$, the left-hand
side of $[w, w']$
 (considered as a prime end) is in $\p D^+$ and the right-hand side is in $\p D^-$. 
  Let $N_+$
be the set of integers $k$ such that the ordered edge
$[ki,ki+1]$ is contained in $\eta$,    $N_-$
be the set of integers $k$ such that the ordered edge
$[ki+1,ki]$ is contained in $\eta$, and $N = N_+ \cup
N_-$.  We claim that if $j \in N_+$
and $k$ is the largest integer less than $j$ with $k \in N$,
then $k \in N_-$.  For otherwise, the open line segment from
$ji + (1/2)$ to $ki + (1/2)$ would be in both $D^+$ and
$D^-$.   We now consider the smallest $k$ such that
$k \in N$.   The line segment from $-ni + (1/2)$ to $ ki + (1/2)$ is
contained in $D^-$ and hence $k \in N_+$.    As we continue
up the line $\{\Re(z) = 1/2\}$ we see that when we intersect
edges in $\eta$, they alternate being in $N_+$ or $N_-$,
the first in $N_+$, the second in $N_-$, the third in $N_+$,
etc.  When we reach the unordered edge $\{0, 1\}$, we see that
if $0 \in N_+$, then there have been an even number of edges
before $\{0,1\}$ and if $0 \in N_-$, there have been
an odd number of edges.   In other words, $(-1)^{J(\eta)} = 1$
if $\eta \in \saws_n^+$ and $(-1)^{J(\eta)} = -1$ if
$\eta \in \saws_n^{-}$.  This gives \eqref{nov19.1}.

 Let $\Lambda_n$ be defined as in \eqref{Lambda}. We claim that 
\begin{equation}  \label{nov19.2}
 \Lambda_n = \sum_{\omega \in \paths_n}
      q(\omega) \, Y(L(\omega))
        = \sum_{\eta \in \saws_n^*} \sum_{L(\omega) = \eta}  p(\omega) \, (-1)^{J
        (\omega)- J(\eta)} . 
           \end{equation}
To see this, 
suppose that
$L(\omega) = \eta = [\eta_0,\ldots,\eta_k]$.  
We can write
\[ \omega =  [\eta_0,\eta_1] \oplus l_1 \oplus [\eta_1,\eta_2]
 \oplus l_2 \oplus \dots\oplus  [\eta_{k-2},\eta_{k-1}]
 \oplus l_{k-1} \oplus [\eta_{k-1},\eta_k],\]
where $l_j$ is a loop rooted at $\eta_j$ that does not enter
$\{\eta_1,\ldots,\eta_{j-1}\}$.  We write
\[     J(\omega) = J(\eta) + J_L(\omega) , \;\;\;\;
    Y(\omega) = Y(\eta) + Y_L(\omega) , \]
  where $J_L,Y_L$ denote the contributions  from  the loops. Then
\[   Y(\omega) = Y(\eta) + \sum_{j=1}^{k-1} Y(l_j).\]
For each loop $l_j$ there is the corresponding reversed loop
$l_j^R$ for which $Y(l_j^R) = - Y(l_j)$.  
 Since $J(l_j^R) = J(l_j)$ and $Y(l_j^R) = - Y(l_j)$,
 we get cancellation.  Doing this for all the loops, we see that
\[    \sum_{\omega \in \paths_n, L(\omega) = \eta} q(\omega) \, [Y(\omega)
 - Y(\eta)] = 0 . \]
 This gives the first equality in \eqref{nov19.2}.  The second equality
uses \eqref{nov19.1} and 
   the fact that $Y(\eta) = 0$
 if $\eta \not \in \saws_n^*$.

If  $\eta \in \saws_n^*$, then
\begin{eqnarray*}
\sum_{L(\omega) = \eta}   {p(\omega)} \, (-1)^{J(\omega)- J(\eta)} 
& = & p(\eta) \sum_{L(\omega) = \eta}  \frac{p(\omega)}{p(\eta)} \, (-1)^{J_L(\omega)} \\
 & = & p(\eta)\,
 Q_\eta(A_n) \\& = & p(\eta)\,\exp \left\{\sum_{\bar l \subset A_n,
 \bar l \cap \eta \neq \emptyset} (-1)^{J(\bar l)} \, m(\bar l) \right\}\\
 & = &p(\eta)\, F_\eta(A_n) \,\exp \left\{-2\sum_{\bar l \subset A_n,
 \bar l \cap \eta \neq \emptyset, J(\bar l) \mbox{ odd}}  m(\bar l) \right\}.
 \end{eqnarray*}
 If $J(\bar l)$ is odd, then $\bar l$ must include at least one
 unordered edge $\{ki, ki+1\}$ with $k \geq 0$ and at least
 one unordered edge $\{ki,ki+1\}$ with $k < 0$. 
Therefore, topological considerations imply that if $\eta \in \saws_n^*$,
   then $\eta 
  \cap \bar l  \neq \emptyset.$
  Hence
  \[  \sum_{\bar l \subset A_n,
 \bar l \cap \eta \neq \emptyset, J(\bar l) \mbox{ odd}}  m(\bar l)
  =  \sum_{\bar l \subset A_n,
  J(\bar l) \mbox{ odd}}  m(\bar l) =   { m(\looper_{A_n})}. \]
   Combining this with \eqref{nov19.2}, we see that
\begin{equation}  \label{nov19.4}
\Lambda_n = \sum_{\eta \in \saws_n^*} p(\eta) \, Q_\eta(A_n) 
      =   e^{-2m(\looper_{A_n})  }    \sum_{\eta \in \saws_n^*} p(\eta) 
             \, F_\eta(A_n) =   e^{-2m(\looper_{A_n})  }  
              \sum_{\eta  \in \saws_n^*}  \hat p_n(\eta) . 
              \end{equation}

We will now compute $\Lambda_n$  as defined
in \eqref{Lambda} a different way. 
Let $\omega = [\omega_0,\ldots,\omega_\tau]
\in \paths_n$. 
 If $\omega$ does not visit $0$ or $\omega$ 
does not visit $1$, then $Y(\omega) = 0$.  Hence,
we need only consider the sum over
 $\omega \in \paths_n$ that visit both $0$
and $1$.  
Let $T_0 = \min\{j: \omega_j = 0\}, T_0' = \max\{j < \tau:
\omega_j = 0\}$ and define $T_1,T_1'$ similarly.

Suppose that $T_0 < T_1, T_0' > T_1'$.  In this case we
write
\begin{equation}  \label{jan9.6}
 \omega = \omega^- \oplus l \oplus \omega^+, 
 \end{equation}
where $l$ is the loop $[\omega_{T_0},\ldots,\omega_{T_0'}]$.
 Note that $Y(\omega) = Y(l)$. 
For any such loop, there is the corresponding reversed loop 
$l^R = [\omega_{T_0'}, \omega_{T_0' - 1},\ldots,
\omega_{T_0}]$ for which $Y(l^R) = - Y(l)$.  These terms cancel
and hence the sum in \eqref{Lambda} over $\omega$
with $T_0 < T_1, T_0' > T_1'$ is zero.  Similarly, the
sum over $\omega$ with  $T_1 < T_0 \leq
T_0' < T_1'$ is zero.

Suppose that $T_0 > T_1, T_0' > T_1'$.  Then
we can write 
$\omega$ as
\[  \omega =
 \omega^- \oplus l_1 \oplus \omega' \oplus  l_0 \oplus \omega^+. \]
Here $l_0$ is a loop rooted at $0$, $l_1$ is a loop 
rooted at $1$, $\omega'$ is a path from $1$
to $0$,  $\omega^-$ is a path from $\{\Re(z) = -n\}$ to $1$ avoiding
$0$,
$\omega^+$ is a path from $0$ to $\{\Re(z) = n+1\}$ avoiding $1$.  Let
$\tilde \omega^-$ be the reflection of $\omega^-$
about the real axis.  Then $J(\omega^-) + J(\tilde \omega^-)
$ is odd and these terms will cancel.  Hence the sum over
all $\omega$ with  $T_0 > T_1, T_0' > T_1'$
is zero.

 Let $ \paths_n^1$ be the set of paths in $\paths_n$
that visit both $0$ and $1$ and satisfy
 $T_0 < T_1, T_0' < T_1'$.  We have shown that
 \[   \Lambda_n = \sum_{\omega \in \paths_n^1}  
  q(\omega) \, Y(\omega) . \] 
 If $\omega \in \paths_n^1$, let
$\rho = \min\{j > T_0':  \omega_j = 1 \}$.  Then we can write
$\omega$ as
\begin{equation}  \label{nov19.5}
  \omega =
  \omega^- \oplus l_0 \oplus \omega' \oplus  l_1 \oplus \omega^+. 
 \end{equation}
Here $l_0$ is a loop rooted at $0$, $l_1$ is a loop 
rooted at $1$, $\omega' =[\omega_{T_0'},
\ldots,\omega_\rho]$ is a path from $0$ to
$1$,  $\omega^-$ is a path from $\{\Re(z) = -n\}$ to $0$,
$\omega^+$ is a path from $1$ to $\{\Re(z) = n+1\}$.  The paths $\omega',
\omega^-,\omega^+$ do not enter $\{0,1\}$ except at their
endpoints.  The loop $l_1$ does not visit $0$.  All points other
than the endpoints must lie in $A_n$.  These are all the
  restrictions on the paths.   Note that $Y(\omega) = Y(l_0) +
  Y(\omega')$. As in the previous arguments, we can replace
  $l_0$ with the reversed loop  $l_0^R$, to see that
  \[  \sum_{\omega \in \paths_n^1}
   (-1)^{J(\omega)} \, Y(l_0) \, p(\omega)  =0.\]
Also $Y(\omega')  \in \{0,1\}$ with $Y(\omega') = 1$ if and
only if $T_0'+ 1 = \rho$, that is, if $\omega' = [0,1]$.
Therefore, if $\paths_n^2$ denotes the set of paths in $\paths_n^1$
with $\omega' = [0,1]$, then  
\begin{equation}  \label{nov20.4}
 \Lambda_n
 = \sum_{\omega \in  \paths_n^2}
          (-1)^{J(\omega )}\, p(\omega)
      =  \sum_{\omega \in  \paths_n^2}
          (-1)^{J(\omega^-) +J(l_0) + J(l_1)
             + J(\omega^+)}\, p(\omega). 
             \end{equation}

If $\omega \in \paths_n^2$, let $\xi$ be the smallest $j$ such that $\omega_j$ is
on the positive real axis.  Suppose for the moment that
$\xi < T_0$.  Then we can write
\[   \omega^- = \omega^{-,1} \oplus \omega^{-,2}, \]
by splitting the path at time $\xi$.  The path $\omega^{-,2}$
is a path from the positive real axis to $0$ that does
not go through the point $1$.  Hence, 
$J(\omega^{-,2}) + J(\tilde \omega^{-,2})$ is odd, where
$\tilde \omega^{-,2}$ is the reflection of
$\omega^{-,2}$ about the
real axis.  These terms will cancel in the sum \eqref{nov20.4}, and hence
it suffices to sum over $\omega^-$ such that $\omega^-
\cap [1,\infty) = \emptyset$.  For these $\omega^-$,
we can see by topological reasons that
 $(-1)^{J(\omega^-)} = 1$.    By a similar argument,
it suffices to sum over $\omega^+$ satisfying
$\omega^+ \cap (-\infty,0] = \emptyset$, 
and for these walks $(-1)^{J(\omega^+)} = 1$. Therefore, if   
$\paths_n^3$ denote the set of paths in $\paths_n^2$ satisfying
 \[  \omega^- \cap [1,\infty) = \emptyset, \;\;\;
 \omega^+ \cap (-\infty,0] = \emptyset, \]
we see that
\[ \Lambda_n =\sum_{\omega \in  \paths_n^3}
          (-1)^{ J(l_0) + J(l_1)
             }\, p(\omega). \]

Let us write any $\omega \in \paths_n^3$ as in \eqref{nov19.5}. 
We must choose $\omega^- \in \paths_n', (\omega^+)^R \in \paths_n''$
and $\omega' = [0,1]$.  Summing over all of these possibilities,
gives a factor of $R_n^2/4$. 
The choices of $l_0,l_1$ are independent of the choices of
$\omega^-$ and $\omega^+$.  The only restriction is that
the loops lie in $A_n$ and $l_1$ does not contain the origin.  By our definition,
\[    \sum_{l_0,l_1} (-1)^{J(l_0) + J(l_1)} \, p(l_0) \,
p(l_1)
   = g_{A_n}(0,0) \, g_{A_n\setminus \{0\}}(1,1) = Q_{01}
    (A_n) . \]
Therefore, 
\[     \Lambda_n =\sum_{\omega \in  \paths_n^3}
          (-1)^{ J(l_0) + J(l_1)
             }\, p(\omega) = \frac 14 \, R_n^2 \, Q_{01}(A_n) . \]
 Comparing this with \eqref{nov19.4} gives the theorem.
\end{proof}

 \section{Estimate on the random walk loop measure} \label{rwloopsec}

Using Theorem \ref{combin} and the estimates
\eqref{nov25.1} and \eqref{nov25.2},
  we see that
\[    \sum_{\eta \in \saws_n^*} \hat p_n(\eta) 
  \asymp n^{-1} \, \exp \left\{2 m(\looper_{A_n})
  \right\}. \] The proof of  \eqref{nov20.1} 
  is finished with the following proposition. 

%

\begin{proposition}
There exists $c < \infty$ such that for all $n$,
\[         \left|  m(\looper_{A_n}) - \frac {1} 8 
 \, \log n\right| \leq c . \]
    \end{proposition}

\begin{proof}   
    Let $C_n = \{z \in \Z^2: |z| < e^n\}. $ We will prove the
    stronger fact that the limit 
\begin{equation}  \label{nov20.2}
       \lim_{n \rightarrow \infty} \left[ m(\looper_{C_n}) - \frac n  8    
    \right]
    \end{equation}
    exists 
  by showing that
\begin{equation}  \label{nov21.7}
 \sum_{n=1}^\infty \left|  m(\looper_{C_{n+1}}\setminus
     \looper_{C_{n}}) - \frac 1 8 \right| < \infty . 
     \end{equation}
     
     Let $\mu$ denote the Brownian loop measure,
      and let
     $\tilde \looper$ denote the set of unrooted  Brownian
     loops $\gamma$ in
     the unit disk that intersect $\{|z| \geq e^{-1}\}$ and such
     that the winding number of $\gamma$ about the origin is odd.
We will establish \eqref{nov21.7}
 by showing  that $\mu(\tilde \looper) = 1/8$ and
   \begin{equation}  \label{nov20.3}
  \left| m(\looper_{C_{n+1}}\setminus
     \looper_{C_{n}}) -\mu(\tilde \looper) \right|   = O(n^{-2}).
   \end{equation}

For the Brownian loop measure, we do a computation similar to that in \cite[Proposition 3.9]{Annulus}.
Using \eqref{loopdefinition},  we write
\[   \mu (\tilde \looper)=
\frac 1 \pi \int_{e^{-1}}^1 \int_{0}^{2 \pi} \phi(r,\theta) \, d\theta \,r\,  dr, \]
where $\phi(r,\theta)$ denotes the
    Brownian bubble measure of loops in
$r \Disk$  rooted at $r e^{i\theta}$ with odd winding number about
the origin.  Rotational symmetry implies that $\phi(r,\theta) = \phi(r,0)$
and conformal covariance implies that $\phi(r,0) = r^{-2}
\,\phi$
where $\phi = \phi(1,0)$.  Hence,
\begin{equation}  \label{jan7.1}
\mu (\tilde \looper)= \frac{\phi}\pi 
 \int_{e^{-1}}^1 \int_{0}^{2 \pi} r^{-2} \, d\theta \,r\,  dr = 2  \phi. 
 \end{equation}
%
By considering the (multi-valued) covering map $f(z) = i\log z $
which satisfies $|f'(1)| = 1$, we see that
\[\phi = \sum_{k \mbox{ odd}}  H_{\p \Half}(0, 2\pi k) , \]
where $H_{\p \Half}$ denotes the boundary Poisson kernel
in the upper half-plane $\Half$  normalized so
that $H_{\p \Half}(0,x) =x^{-2}$.  Therefore,
\[      2\phi   = 2 \sum_{ k=-\infty}
^ \infty  \frac{1}{[2\pi (2k+1)]^2} = \frac{1}{\pi^2} \,  \left[1
 + \frac 1{3^2} + \frac 1{5^2} + \cdots\right] = \frac 1
 {8} . \]

If $s > 2$, let  $\tilde  \looper _s^*$ denote the set of 
Brownian loops in $\Disk$ that intersect
both $\{|z| \geq e^{-1} \}$ and $\{|z| \leq e^{-s}\}$.   
We claim that as $s \rightarrow \infty$,
 \begin{equation}  \label{nov21.6.alt} 
       \mu(  \tilde \looper_s^*) =    s^{-1} + O(s^{-2}),  
\end{equation}
\begin{equation}  \label{nov21.6} 
     \mu [\tilde \looper \cap \tilde \looper_s^*]
     = (2s)^{-1}+ O(s^{-2} ) . 
    \end{equation}
To see this, we first consider the boundary bubble measure $\lambda_s$ of loops
in $\Disk$ rooted at $1$ that enter $\{|z| \leq e^{-s}\}$.  An exact
expression is given as follows.  Let $B_t$ be a Brownian motion
and $\sigma_s = \inf\{t: |B_t|= e^{-s}\}$.   Then, 
\[     \lambda_s = \lim_{\epsilon \downarrow 0} \E^{\epsilon}
  \left[H_\Disk(B_{\sigma_s},1); \sigma_s < \sigma_0 \right] . \]
The Poisson kernel in the disk is well known; for our purpose we need
only know that
\[    H_\Disk(z,1) = \frac 12 + O(|z|), \]
and a standard estimate for Brownian motion gives
 \[ \Prob^{1-\epsilon} \{\sigma_s < \sigma_0\} = \frac{\log (1-\epsilon  )}
 {-s} \sim \frac{\epsilon}{s}. \]
 Therefore, $\lambda_s  = (2s)^{-1} \, + O(e^{-s})$. 
 Using rotational invariance, and conformal convariance, if $r \geq e^{-1}$
 and  $\lambda(r,\theta,s)$
 denotes the bubble measure of bubbles in $r\, \Disk$ rooted
 ant $re^{i\theta} $ that enter $\{|z| \leq e^{-s}\}$, then
 \[   \lambda(r,\theta,s) = r^{-2} \, (2s)^{-1} \, [1 + O(s^{-1})]. \]
 If we compute as in \eqref{jan7.1}, we get \eqref{nov21.6.alt}.  The
 relation \eqref{nov21.6} is done similarly except that we have to worry
 about the winding number of the loop.  This gives a factor of $1/2$. 
 Note that we have
\begin{equation}  \label{jan7.3}
  \mu [\tilde \looper \cap \tilde \looper_s^*] = \frac 12 
 \, \mu [ \tilde \looper_s^*] \, [1+O(s^{-1})].
 \end{equation}

%

  For each unrooted
random walk loop
$\bar l \in \looper_{C_{n}} \setminus \looper_{C_{n-1}}$,
there is a corresponding continuous unrooted loop $\bar l^{(n)}$
in $\Disk$ obtained from linear interpolation
and  Brownian scaling.
We will write  $d(\bar l,\gamma)
\leq \delta$, if we can parametrize and root  the loops $ \bar l ^{(n)}$ and
 $\gamma$  such that the loops
are within $\delta$ in the supremum norm.  In \cite{Jose}
it was shown  that there exists $\alpha > 0$  
and a coupling of  the random walk and Brownian
loop measures in $D$, restricted to loops of diameter at least
$1/e$, so that the total masses agree up to 
 $O(e^{-n \alpha})$ and such that in the coupling,
except for a set of paths of size $O(e^{-n \alpha})$, we
have $d(\bar l,\gamma) < e^{-n \alpha}$.  We would like
to say that in the coupling,  the Brownian loop 
has odd winding number if and only $J(\bar l)$ is odd.
If the loops stay away from the origin, this holds.
However, if the loops are near the origin, it is possible
for the winding numbers of the continuous and the discrete
walks to be different.  However, and this is why we can prove
what we need, it is also true that if a macroscopic
loop (either continuous or discrete) gets close to the origin,
 then it is just about equally likely to have an odd as
an even winding number.  Let us be more precise.

Let $\beta < \alpha$ and let $\looper^n$ denote the
set of loops in $\looper_{C_{n+1}} \setminus \looper _{C_n}$
that intersect $\{|z| \leq e^{-\beta n} \, e^{n+1}\}$
Using the coupling and $\eqref{nov21.6},$ we see that
\[   m (\looper^n)
 =  \mu(\tilde \looper_{\beta n}') + O(n^{-2}) = (\beta n)^{-1} + O(n^{-2}). \]

   Let us split these
paths into two sets: those for which $\dist(0,\gamma) \leq 2e^{-n\alpha}$
and those for which $\dist(0,\gamma) > 2e^{-n \alpha}$.  If
  $\dist(0,\gamma) >    2e^{-n\alpha}$ and $d(\bar l,\gamma) \leq e^{-n \alpha}$, 
  then $J(\bar l) $ is odd if and only if 
the winding number of $\gamma$  is odd.  Therefore
\[   m((\looper_{C_{n+1}} \setminus \looper _{C_n} ) \setminus
 \looper^n) = \mu(\tilde \looper \setminus \tilde \looper'_{\beta n} )
  + O(n^{-2}) . \]
(The error term $O(n^{-2})$ is comparable to the measure of loops $\gamma$
such that $e^{-n \beta} \leq \dist(0,\gamma) \leq  2e^{-n\beta}$.)  
 
A coupling argument can be used to give a random walk analogue
of \eqref{jan7.3},
\[    m \left[\looper^n \cap (\looper_{C_{n+1}}\setminus \looper_{C_{n}})\right]
   = \frac 12 \,   m (\looper^n) \, [1 + o(n^{-1})]. \]   We sketch the proof.
We use the definition of the loop
measure using an enumeration of $\Z^2 = \{z_1,z_2,\ldots\}$
such that $|z_j|$ increases.  Then an unrooted loop
in $\looper^n$ is obtained
from a loop rooted in $C_{n+1} \setminus C_n$. Let
us call the root $z_k$ and so the loops lies
in $V_k = \{z_1,\ldots,z_k\}$.   Let
us stop the walk at the first time it reaches
a point, say $z'$,
 in $\{|z| \leq e^{-\beta n}\, e^{n+1}\}$.  The
remainder of the loop acts like a random 
walk started at $z'$ conditioned
to reach $z_k$ without before leaving $V_k$. 
Let $J'$ denote the number of times such a walk
crosses the half line $\{(1/2) +iy :  y < 0\}.$
We claim that the probability that $J'$
is odd equals $\frac 12 + O(e^{-u n})$ for
some $u > 0 $.   Indeed, we can couple two walks
starting at the point so that  
each walk has the distribution
of random walk conditioned to reach $z_k$ before
leaving $V_k$ and that, except for an event
of probability $O(e^{-\delta})$, the parity of $J'$
is different for the two walks. This uses a standard
technique.  The key estimate is the following.
There exists $c > 0$ such that if $S$ is a simple
random walk starting at $z \in C_{j-1}$ and
$T = \min\{j: S_y \in C_j\}$, then 
for all $w \in \p C_j$ with $\Im(w) > 0$,  
\[    \Prob\{S(T) = w  , \; J' \mbox{ odd} \} \geq
     c \, e^{-j} , \]
\[ \Prob\{S(T) = w  , \; J' \mbox{ even} \} \geq
     c \, e^{-j} . \]
Without the restriction of the parity of $J'$, see,
for example,
\cite[Lemma 6.3.7]{LLimic}.  To get the result with the restriction, we just
note that there is a positive probability of making
a loop in the annulus $C_{j} \setminus C_{j-1}$, and
this increases $J'$ by one. Hence, we can find a coupling
and a $\rho > 0$ such that at each annulus there is
a probability $\rho$ of a successful coupling given that
the walks have not yet been coupled.  Since there
are of order $  {\beta} \,  n$ annuli,
we can couple the processes so that
 the probability
of not being coupled  is $(1-\rho)^{\beta n} =
O(e^{-u n})$ for some $u$.

  From the last two estimates and \eqref{nov21.6}, we see that
\[   \left |\mu(\tilde \looper)- m(\looper_{C_{n+1}}\setminus
     \looper_{C_{n}})\right|  \leq c \, n^{-2}  . \]
       This gives \eqref{nov20.3}.

 \end{proof}
 
 \section{Different boundary conditions}
 
 In this section, we generalize 
  Theorem \ref{combin} to more general
 boundary conditions.  Let $\bar \zeta =(\zeta_1,\zeta_2)$ be
 an ordered pair of distinct points in 
 $\p A_n$, and let $\paths_n(\bar \zeta)$ to be the
 set of nearest neighbor paths $\omega =[\omega_0,\ldots,\omega_k]$
 with $\omega_0 = \zeta_1, \omega_k \in \zeta_2$ and
 $\{\omega_1,\ldots,\omega_{k-1}\} \subset A_n \setminus
 \{0,1\}$.   Let
 $\saws_n^+(\bar \zeta), \saws_n^-(\bar \zeta), 
 \saws_n^*(\bar \zeta)$ be defined similarly  
 with the new boundary condition.   
 As before, if  $\eta \in \saws_n^*(\bar \zeta)$, then
 \[   \hat p_n(\eta) = \sum_{\omega \in\paths_n(\bar \zeta),
 L(\omega) = \eta} \hat p(\omega). \]
 
 Define $R_n(\zeta_j)$ as follows.  Let $S_n$ be a simple random walk
 starting at the origin and let
 \[    T = T_n = \min\{j \geq 1: S_j \in [0,\infty) \cup \p A_n \}. \]
 Then
 \[        R_n(\zeta_j) = I(\zeta_j) \, \Prob\{S_T = \zeta_j\}, \]
 where $I(\zeta_j) = -1$ if $\Re(\zeta_j) > 0, \Im(\zeta_j) < 0$ and
 $I(\zeta_j) = 1$ otherwise.  Note that if
 $\omega \in K_n(\zeta_1)$ with $\omega \cap
 [1,\infty) = \emptyset$, then $I(\zeta_1) =
 (-1)^{J(\zeta_1)}.$ Define
 \[     \Phi_n(\bar \zeta) = \left| R_n(\zeta_1) \, R_n(1 - \zeta_2)
     - R_n(1-\zeta_1) \, R_n(\zeta_2) \right|. \]
The combinatorial identity takes the following form.

\begin{theorem}  \label{jan10theorem}
  If $\bar \zeta =(\zeta_1,\zeta_2 ) \in \p A_n \times \p
  A_n$, then
\begin{equation}  \label{jan9.1}
  \sum_{\eta \in \saws_n^*(\bar \zeta)}  \hat p(\eta)
    =   \frac 14 \,  {Q_{01}(A_n) \, \exp\{2 m(\looper_{A_n})\}}
         \, \Phi_n(\bar \zeta) .
         \end{equation}
\end{theorem}  

Note that the factor  ${Q_{01}(A_n) \, \exp\{2 m(\looper_{A_n})\}}/4$
does not depend on $\bar \zeta$. Hence the theorem implies that
if $\bar \zeta, \bar \zeta' \in \p A_n\times \p A_n$,
\[ \frac{ \sum_{\eta \in \saws_n^*(\bar \zeta)}  \hat p(\eta)}
    { \sum_{\eta \in \saws_n^*(\bar \zeta')}  \hat p(\eta)}
      = \frac{\Phi_n(\bar \zeta)} {\Phi_n(\bar \zeta')}.\]

\begin{proof} Let $J(\omega),Y(\omega)$ be as in the introduction.
Both sides of \eqref{jan9.1} vanish if $\zeta_1 = \zeta_2$
so we assume that $\zeta_1 \neq \zeta_2$.
Both sides of \eqref{jan9.1} are invariant
under the transformation $(\zeta_1,\zeta_2) \mapsto (\zeta_2,\zeta_1).$
Let 
\begin{equation}  \label{Dndef}
D_n  = \{x+iy \in \C: -n < x < n+1, -n < y < n\}.
\end{equation}
The boundary $\p D_n$ is a rectangle.  Without loss of generality
we will assume that the order $(\zeta_1,\zeta_2)$ is chosen so that  the point
$\frac 12 -ni$ is on the arc of $\p D_n$
 from $\zeta_1$ to $\zeta_2$ in the
counterclockwise direction.   If $\eta \in \saws_n^*(\bar \zeta)  $,
  considered as a simple curve by linear
interpolation, then $D_n \setminus \eta$ consists of two components,
$D^+,D^-$ where $D^-$ denotes the component containing $\frac 12 -ni$ 
on its boundary.  By our choice of order of $\zeta_1,\zeta_2$, as we
traverse $\eta$ from $\zeta_1$ to $\zeta_2$,
 the right-hand side of $\eta$ is on 
 $\p D^-$ and the left-hand side is  $\p D^+$.   
Using this, we can use the topological argument as in \eqref{nov19.1}
to conclude that
\begin{equation}  \label{jan9.4}
(-1)^{J(\eta)} \, Y(\eta) = 1 \;\;\;\; \mbox{if} \;\;\;\;
  \eta \in \saws_n^*(\bar \zeta).
  \end{equation}

Let
\begin{equation}
\label{jan9.5}
   \Lambda  = \Lambda_n(\bar \zeta)
     = \sum_{\omega \in \paths_n(\bar \zeta)} (-1)^{J(\omega)}
      \, p(\omega) \, Y(\omega) .
      \end{equation}
  We claim that
\[ \Lambda = \exp \left\{-2 m(\looper_{A_n}) \right\}
   \sum_{\eta \in \saws_n^*(\bar \zeta)} \hat p_n (\eta) . \]
   Indeed, given \eqref{jan9.4},
   the proof is identical to that  of \eqref{nov19.4}.
  Hence to prove the theorem, it suffices to prove that
 \begin{equation}  \label{jan9.3}
         \Lambda = \frac{Q_{01}(A_n)}{4} \, \Phi_n(\bar \zeta).
 \end{equation}

As in the proof of Theorem \ref{combin}, we note that if
$\omega$ does not contain the unordered edge $\{0,1\}$,
then $Y(\omega) = 0$.  Hence the sum in \eqref{jan9.5}
is over $\omega$ such that $T_0,T_0',T_1,T_1'$, as
defined in the proof of Theorem \ref{combin},
 are
well defined.   As in that proof, if $T_0 < T_1$
and $T_0' > T_1$,  we can write $\omega$ as
in \eqref{jan9.6}.  By considering the path with the reversed
loop, we see that the sum over $\omega$ with
$T_0 < T_1$
and $T_0' > T_1$ equals zero.  Similarly, the
sum over $\omega$ with $T_0 > T_1, T_0' < T_1'$
is zero.   Hence we need only consider the sum over
paths with $T_0 < T_1, T_0' < T_1'$ or
$T_0 > T_1, T_0' > T_1'$.  However, unlike in the proof
of Theorem \ref{combin}, the sum over $\omega$ with
$T_0 > T_1, T_0' > T_1'$ does not vanish.

Let $\paths^5 = \paths^5_n(\bar \zeta)$ be the set of
$\omega \in \paths_n(\bar \zeta)$ such that $T_0 < T_1$
and $S_{T_0'+1} = 1$.  Let $\paths^6 = \paths^6_n(\bar \zeta)$
be the set of $\omega \in \paths_n(\bar \zeta)$ such that $T_1 < T_0$
and $S_{T_0-1} = 1$.   Then the argument leading to \eqref{nov20.4}
in this case gives
\[  \Lambda = \sum_{\omega \in \paths^5} (-1)^{J(\omega)}
 \, p(\omega) - \sum_{\omega \in \paths^6} (-1)^{J(\omega)}
 \, p(\omega) .\]
 
 If $\omega=[\omega_0,\ldots,\omega_\tau]
  \in \paths^5$, we write
 \[   \omega = \omega_- \oplus l_0 \oplus [0,1] \oplus l_1
   \oplus \omega_+ \] where
   \[ \omega_- = [\omega_0,\ldots,\omega_{T_0}], \;\;\;\;
     l_0 = [\omega_{T_0},\ldots,\omega_{T_0'}], \]
     \[ l_1 = [\omega_{T_0'+1},\ldots,\omega_{T_1'}], \;\;\;\;
       \omega_+ = [\omega_{T_1'},\ldots,\omega_\tau].\]
  Note that $\omega_-$ is a path from $\zeta_1$ to $0$;
  $l_0$ is a loop rooted at $0$; $l_1$ is a loop rooted at $1$
  that avoids the origin; and $\omega_=$ is a path from $1$
  to $\zeta_2$.  The points in $l_0,l_1$ must lie in $A_n$.
  The points in $\omega_-,\omega_+$ other than the endpoints
  must lie in $A_n \setminus \{0,1\}$.   Let $\paths^6 = \paths^6_n(\zeta_1)$
  be the set of nearest neighbor paths $\omega = [\omega_0,\ldots,\omega_k]$
  with $\omega_0 = \zeta_1, \omega_k = 0, \{\omega_1,\ldots,\omega_{k-1}\}
   \subset A_n \setminus \{0,1\}$.  Let $\paths^7 =
   \paths^7_n(\zeta_2)$
    be the set of nearest neighbor paths $\omega = [\omega_0,\ldots,\omega_k]$
  with $\omega_0 =  1, \omega_k = \zeta_2, \{\omega_1,\ldots,\omega_{k-1}\}
   \subset A_n \setminus \{0,1\}$.  Then
   \[ \sum_{\omega \in \paths^5} (-1)^{J(\omega)}
 \, p(\omega) = \frac{Q_{0,1}(A_n)}{4} \, \left[\sum_{\omega \in \paths^6
  } (-1)^{J(\omega)} \, p(\omega)\right] \, \left[\sum_{\omega \in \paths^7}
   (-1)^{J(\omega)}\, p(\omega)\right]. \]
   
   If $\omega = [\omega_0,\ldots,\omega_k] \in \paths^6$, let $j$ be the first
   index (if it exists) such that $\omega_j \in [1,\infty)$.  If $j$ exists, we
   can write
   \[   \omega = [\omega_0,\ldots,\omega_j] \oplus [\omega_j,\ldots,\omega_k].\]
   The path $\omega'= [\omega_j,\ldots,\omega_k]$ is a path from the positive real
   line to the origin.  If $\tilde \omega'$ denotes the reflection of $\omega$
   about the real axis, then $J(\omega') + J(\tilde \omega')$ is odd.  Hence,
   these terms will cancel in the sum.  Therefore,
  \[  \sum_{\omega \in \paths^6
  } (-1)^{J(\omega)} \, p(\omega) = \sum_{\omega \in \paths^8
  } (-1)^{J(\omega)} \, p(\omega), \]
where $\paths^8$ denotes the set of paths in $\paths^6$ that do not
intersect $[1,\infty)$.   If $\omega \in \paths^8$, then $\omega$
cannot hit the positive real line, so we can see that $J(\omega)$
is even unless $\omega$ lies in
the quadrant $\{ \Re(\zeta_1) > 0,\Im(\zeta_1) < 0\}$ in which case
$J(\omega)$ is odd.  
Therefore,
\[ \sum_{\omega \in \paths^8
  } (-1)^{J(\omega)} \, p(\omega) = R_n(\zeta_1). \]
Using symmetry about the line $\{\Re(z) = 1/2\}$, we can see that
\[   \sum_{\omega \in \paths^7}
   (-1)^{J(\omega)}\, p(\omega) = R_n(1 - \zeta_2). \]
Similarly, we see that 
\[ 
\sum_{\omega \in \paths^6} (-1)^{J(\omega)}
 \, p(\omega) =   \frac{Q_{01}(A_n)}{4} \,R_n(\zeta_2) \, R_n(1-\zeta_1).\]
\end{proof}

\section{Scaling limit}

The scaling limit of the loop-erased walk is known in some sense to
be the Schramm-Loewner evolution with parameter $2$ ($SLE_2$),
see \cite{LSW}. 
Here we would like to give a stronger conjecture about this convergence
that has not been proved and then show how the estimates here
could be used to  prove one case of 
this conjecture.  Although we could probably given the details
of the proof, we choose not to bother since it is only a start to the main
result we hope to obtain later.

We start by reviewing $SLE_\kappa$ in the simple curve case.
We take a ``partition function'' view with curves parametrized
by natural parametrization as outlined in \cite{Annulus}.
Suppose $0 < \kappa \leq 4$ and let $a = 2/\kappa \in [1/2,\infty)$.
We will consider bounded simply connected domains $D \subset \C$
containing the origin with piecewise anayltic boundaries.  We will consider
what is sometimes called two-sided radial $SLE_\kappa$, or, more
convenient for use, as chordal $SLE_\kappa$ from $z$ to $w$ 
conditioned to go through the origin.    This is a finite
measure on simple curves
$\gamma:[0,T_\gamma] \rightarrow \overline D$  that go through
the origin with $\gamma(0) = z, \gamma(T_\gamma) = w, \gamma(0,T_z)
\subset  D$.  We denote this measure by $\mu_D(0;z,w)$ and write
it as
\[       \mu_D(0;z,w)=  \tmass_D(0;z,w) \, \mu_D^\#(0,z,w) \]
where $\tmass_D(0;z,w)$ denotes the total mass of the  measure
and  $\mu_D^\#(0,z,w) $ is a probability measure.   The measure is
supported on curves of Hausdorff dimension $d = d_\kappa =
1 + \frac \kappa 8$.   The family
of  measures satisfies a conformal
covariance property that we now describe.  Suppose $F:D \rightarrow F(D)$
is a conformal transformation with $F(0) = 0$ and such that 
$\p F(D)$ is locally analytic near $F(z),F(w)$.  If $\gamma$ is
a curve, define the curve $F \circ \gamma$ to be the image of
the curve reparametrized so that the time to traverse
$F(\gamma[r,s])$ is
\[            \int_r^s |F'(t)|^d \, dt . \]
If $\nu$ is a measure on curves on $D$, we define $F \circ \nu$
to be the measure on cureves on $F(D)$,
\[      F\circ \nu(V) = \nu\{\gamma: F \circ \gamma \in V \}. \]
Then the conformal covariance rule is
\[         F \circ  \mu_D(0;z,w) = |F'(z)|^b \, |F'(w)|^b \,
   |F'(0)|^{2-d} \,  \mu_{F(D)} (F(0);F(z), F(w)), \]
   where $b = 3a-1 = \frac{6 - \kappa}{2\kappa}$. 
  This rule can be considered as a combination of two rules,
\begin{equation}  \label{tmassrule}
   \tmass_D(0;z,w) = |F'(z)|^b \, |F'(w)|^b \,
   |F'(0)|^{2-d} \,\tmass_{F(D)}(0;F(z),F(w)), 
   \end{equation}
   \[       F \circ  \mu_D^\#(0;z,w) =  \mu_{F(D)}^\# (F(0);F(z), F(w)). \] 
   For each $\kappa$ there is a unique such family
   of measures  up to two arbitrary multiplicative constants.  (If $\mu_D$
   satisfies  the conditions above, then so does $c \, \mu_D$, and also the
   measure obtained by changing the unit of time,
   that is,  replacing $\gamma(s), 0 \leq s \leq t$
   with $\tilde \gamma (s) = \gamma(sr), 0 \leq s \leq t/r$.) 
For simply connected $D$,  
the total mass is given (up to multiplicative constant) by
\[    \tmass_D(0;z,w) = H_{\p D}(z,w)^b \, G_{D}(0;z,w)\]
where $G_D(0;z,w)$ denotes the $SLE_\kappa$ Green's function
giving the normalized probability that a chordal $SLE_\kappa$ path
from $z$ to $w$ goes through the origin.  Up to multiplicative
constant, 
\[            G_\Disk(0;1,e^{2i\theta}) = |\sin \theta|^{4a-1} . \]
and for other domains can be determined by
\[   G_D(0;z,w) = |F'(0)|^{2-d} \, G_{F(D)}(F(z),F(w)). \]

Loop-erased walk corresponds to $\kappa = 2\, (a=1,b=1,d = 5/4)$,
and for the remainder of this section we fix $\kappa = 2$. 
If $U,V$ are closed subarcs of $\p D$, we define
\[       \tmass_D(0;U,V) = \int_U \int_V \tmass_D(0;z,w) \, |dz|
  \, |dw|, \]
  \[       \mu_D(0;U,V) = \int_U \int_V \mu_D(0;z,w) \, |dz|
  \, |dw|. \]
  Since $b=1$, the scaling rule \eqref{tmassrule} can be used to see
  that
  \[     \tmass_D(0;U,V) = |F'(0)|^{3/4} \, \tmass_{F(D)}(0;F(U),
  F(V)). \]
  Given this, we can extend the definition to domains with
  rough boundaries.  
  
  We will now state a {\em conjecture} about the convergence of
  loop-erased walk to $SLE_2$.  We will state it for simply connected domains,
  but we expect it to be true for finitely connected domains as well.
  \begin{itemize}
  \item  Suppose $D$ is a bounded, simply connected domain
  containing the origin whose boundary $\p D$ is a Jordan
  curve.
  \item  Let $U,V$ be disjoint, closed subarcs of $\p D$.
  \item  Let $\sdom = \{x+iy \in \C: |x|,|y| \leq 1/2\}$
  and if $z \in \Z\times i \Z$, let  
  $\sdom_z = z + \sdom$. 
  \item  Let $A_n = A_n(D)$ be the connected component containing
  the origin of the set of $z \in \Z \times i \Z$ such that
  $\sdom_z \subset nD$.
  \item  Let $U_n$ be the set of $z \in \p A_n$ such that
  $\sdom_z \cap nU \neq \emptyset$.  Define $V_n$ simiarly.
  \item  Let $\paths_n = \paths_n(D,U,V)$ denote the set of
  nearest neighbor paths $\omega = [\omega_0,\ldots,\omega_k]$
  with $\omega_0 \in U_n, \omega_k \in V_n$ and $\{\omega_1,
  \ldots,\omega_{k-1}\} \subset A_n$.  Let $\saws_n^* =
   \saws_n^*(D,U,V)$ denote the set of self-avoiding paths in $\paths_n$
   that contain the unordered edge $\{0,1\}$.  As before, let
   \[   \hat p_n(\eta) = \sum_{\omega \in \paths_n, L(\omega)
    = \eta} p(\omega).\]
  \item If $\eta = [\eta_0,\ldots,\eta_k] \in \saws_n^* $,
  define the scaled path $\eta^{n}(t)$ by
  \[        \eta^{(n)}(jn^{-5/4}) = \eta_j/n, \;\;\; j=0,1,\ldots,n \]
  and linear interpolation in between.
  Let $\mu^{(n)} = \mu^{(n)}_D(0;U,V)$ denote the measure of 
  paths that gives measure $ \hat p_n(\eta)$ to $\eta^{(n)}$ for
  each $\eta \in  \saws_n^*(D,U,V)$.
  \end{itemize}
  
  \begin{conjecture} We can choose the two arbitrary
  constants in the definition of the measure $\mu_D$
   such for every $D,U,V$ as above,
\[    \lim_{n \rightarrow \infty} n^{3/4} \, \mu^{(n)} =
       \mu_D(0;U,V) . \]
       In particular, 
\begin{equation}  \label{jan20.5}
 \lim_{n \rightarrow \infty} \sum_{\eta \in \saws_n}
   \hat p_n(\eta) = \tmass_D(0;U,V). 
 \end{equation}
   \end{conjecture}

We will show how Theorem \ref{jan10theorem} can be interpreted
as one case of the conjecture \eqref{jan20.5}. 
 Let $\Disk$ denote the unit
disk in $\C$ and $D$ the square
\[       D = \{x+iy: |x|,|y| < 1 \}. \] If
  $z \in \p D$, let $z_n$ be the corresponding point in $\p A_n$.
  (Approximately $z_n = n z$, but we need to round to the nearest
  integer and compensate for the fact that $A_n$ is not exactly
  a square centered at the origin.  These are very small errors.)
  Then, we would expect from \eqref{jan9.1} that
\begin{equation}  \label{jan10.2}
   n\, \Phi_n(z_n,w_n ) \sim c \, \tmass_D(0;z,w)
    \sim   c'\,   H_{\p D} ( z,w) \,
  G_D(0;z,w)  .
  \end{equation}
We will show why this holds.   We will not calculate the
right-hand side explicitly (we could by computing the map $F$ below,
but it will not be necessary).  However, we do know that
\[ H_{\p \Disk} ( e^{2i\theta_1}, e^{2i\theta_2}) \,
  G_D(0;e^{2i\theta_1}, e^{2i\theta_2})   \sim c \, [\sin (\theta_1-\theta_2)]^{-2}
   \, |\sin (\theta_1 - \theta_2)|^3 = c \, |\sin(\theta_1 - \theta_2)|.
   \]

There is a unique conformal transformation
\[            F: \Disk \rightarrow D, \]
with $F(-1) = -1, F(0) = 0, F(1) = 1$.   One way to show it
exists is to define $F: \Disk \cap \Half \rightarrow
D \cap \Half$ to be the unique conformal transformation
that fixes the boundary points $-1,0,1$, and then to extend $F$
to $\Disk$ by Schwarz reflection. Let $\Disk^+
= \Disk \setminus (-1,0], \Disk^- =\Disk \setminus [0,1),
D^+ = D \setminus (-1,0], D^- = D \setminus [0,1)$, and note
that $F$ conformally maps $\Disk^+$ onto $D^+$ and
$\Disk^-$ onto $D^-$.   Using conformal invariance, we see
that
  \[          G_D(0;F(e^{2i\theta_1}),F(e^{2i\theta_2})  )=
          |F'(0)|^{-3/4} \,   G_\Disk(0;e^{2i\theta_1},e^{2i\theta_2}) =
               c_3 \,  |\sin(\theta_2-\theta_1)|^3, \]
  \[   H_{\p D} (F(e^{2i\theta_1}),F(e^{2i\theta_2}))\,
  G_D(0;F(e^{2i\theta_1}),F(e^{2i\theta_2})  ) = \hspace{1in}\]
\begin{equation}  \label{jan10.3}
\hspace{1in}   c_4\, |F'(e^{2i\theta_1})|^{-1} \, |F'(e^{2i\theta_2})|^{-1}\, 
  |\sin(\theta_2-\theta_1)| . 
  \end{equation}

Let $g_+$ denote the Poisson kernel in the slit domain $\Disk^+$
defined as follows.  Start a Brownian motion near the origin in
$\Disk^+$ and condition the Brownian motion to exit $\Disk^+$ 
on the unit circle $\p \Disk$.  Then $g_+$ is the conditional
density of the exit distribution (taken in the limit as the 
starting point approaches the origin).  It is not difficult to compute
$g_+$ using conformal invariance
\[         g_+(e^{2i\theta}) = c_1 \, \sin \theta . \]
(In this section, we will use $c_1,c_2,\ldots$ for   absolute
constants whose values could be computed, but we will not bother
to.)  If $g_-$ is the corresponding density in $\Disk^-$,
\[        g_-(e^{2i\theta}) =  c_1 \, \cos \theta . \]
Let $\hat g_+, \hat g_-$ denote the corresponding densities
for the slit domains $D^+, D^-$.  Conformal covariance implies
that
\begin{equation}  \label{jan10.1}
              \hat g_+(F(e^{i\theta})) =
 c_2 \, |F'(e^{2i\theta})|^{-1} \, \sin \theta, \;\;\;\;
      \hat g_-(F(e^{i\theta})) =
 c_2 \, |F'(e^{2i\theta})|^{-1} \, \cos \theta.
 \end{equation}

%
%
%
 Suppose that $z = F(e^{2i\theta_1}), w = F(e^{2i\theta_2}). $  
The relation \eqref{jan10.1} suggests (and, in fact, it can be proved)
that
\[       | R_n(z_n) | \sim c_6 \, n^{-1/2} \, |F'(e^{2i\theta_1})|^{-1}
 \,  \sin \theta_1, \;\;\;\; |R_n(1-z_n)| \sim
    c_6 \, n^{-1/2} \, |F'(e^{2i\theta_1})|^{-1}
 \, \cos \theta_1 . \]
 (Here we use the fact that $|F'(x+iy) = |F'(-x+iy)|$.)
Let us consider two cases.  First, suppose that $z,w$ are both
in the upper half plane, that is, $0 \leq \theta_2 < \theta_1 \leq \pi/2$.
Then, 
\begin{eqnarray*}
n\, \Phi_n(z_n,w_n) & = &  n \, \left[R_n(z_n) \, R_n(1-w_n)
  - R_n(1-z_n ) \, R_n(w_n) \right]\\
   & \sim & c_6^2  |F'(e^{2i\theta_1})|^{-1} \,
   |F'(e^{2i\theta_1})|^{-1} \, \left[\sin \theta_1 \, \cos \theta_2
    - \cos \theta_1 \, \sin \theta_2\right]\\
   & = &  c_6^2  |F'(e^{2i\theta_1})|^{-1} \,
   |F'(e^{2i\theta_1})|^{-1} \, \sin(\theta_1-\theta_2) 
  \end{eqnarray*}
As a second case, suppose that
$z = F(e^{2i\theta_1}), w = F(e^{-2i\theta_2}). $   
where we also assume that $\theta_2 \leq \pi/4$ so that
  $w_n$ is in the southeast
quadrant.  Then $R_n(w_n) < 0, R_n(1-w_n) > 0$,
and
\begin{eqnarray*}
n\, \Phi_n(z_n,w_n) & = &  n \, \left[R_n(z_n) \, R_n(1-w_n)
  - R_n(1-z_n ) \, R_n(w_n) \right]\\
   & \sim & c_6^2  |F'(e^{2i\theta_1})|^{-1} \,
   |F'(e^{2i\theta_1})|^{-1} \, \left[\sin \theta_1 \, \cos \theta_2
 +\cos \theta_1 \, \sin \theta_2\right]\\
   & = &  c_6^2  |F'(e^{2i\theta_1})|^{-1} \,
   |F'(e^{2i\theta_1})|^{-1} \, \sin(\theta_1 +\theta_2) 
  \end{eqnarray*}
If we compare this with \eqref{jan10.3}, we see that we
get the prediction \eqref{jan10.2}.

\end{document}